\documentclass[11pt]{amsart}

\usepackage[T1]{fontenc}
\usepackage[utf8]{inputenc}
\usepackage[margin=1in]{geometry}
\usepackage{amsmath,amssymb,amsthm}
\usepackage{mathrsfs}
\usepackage{enumitem}
\usepackage{tikz}
\usetikzlibrary{shadings,patterns,arrows.meta,calc,decorations.pathreplacing}
\usepackage{hyperref}

\theoremstyle{plain}
\newtheorem{theorem}{Theorem}
\newtheorem{proposition}{Proposition}
\newtheorem{lemma}{Lemma}
\theoremstyle{definition}
\newtheorem{definition}{Definition}
\newtheorem{example}{Example}
\theoremstyle{remark}
\newtheorem{remark}{Remark}
\DeclareMathOperator{\Supp}{Supp}
\newcommand{\C}{\mathbb{C}}
\newcommand{\R}{\mathbb{R}}
\newcommand{\N}{\mathbb{N}}
\newcommand{\Oh}{\mathcal{O}}

\newcommand{\rlct}{\operatorname{lct}}
\newcommand{\lct}{\operatorname{lct}}
\newcommand{\ord}{\operatorname{ord}}
\newcommand{\Vol}{\operatorname{Vol}}

\newcommand{\Exc}{\mathrm{Exc}}
\newcommand{\Cont}{\mathrm{Cont}}

\DeclareMathOperator{\Jac}{Jac}

\DeclareMathOperator{\codim}{codim}
\DeclareMathOperator{\Spec}{Spec}
\DeclareMathOperator{\Hom}{Hom}

\begin{document}

\title[Divisorial Geometry of Volume Asymptotics]{On the Divisorial Geometry of Volume Asymptotics of Sublevel Sets}
\thanks{The author acknowledges the support Funda\c{c}\~{a}o de Amparo \`{a} Pesquisa do Estado de S\~{a}o Paulo - Brazil (FAPESP), grant no.~2019/21181--0, and the Coordena\c{c}\~ao de Aperfei\c{c}oamento de Pessoal de N\'ivel Superior -- Brazil (CAPES), MATH-AmSud program, Grant No.\ 88881.179491/2025-01.}

\author[N. Grulha]{Nivaldo Grulha}
\address{Instituto de Ci\^encias Matem\'aticas e de Computa\c{c}\~ao,
Universidade de S\~ao Paulo, S\~ao Carlos, SP, Brazil}


\date{}

\begin{abstract}
The log canonical threshold (LCT) is a fundamental invariant in
birational geometry and singularity theory, measuring the complexity of
an analytic singularity through discrepancy and valuation data on a log
resolution.

In this work we investigate the asymptotic behaviour of sublevel-set
volumes associated with principal analytic ideals, equivalently with
holomorphic function germs. Building on the classical
theory of local zeta functions and Mellin asymptotics, we introduce the \emph{visible
spectrum} — the set of actual poles of the local zeta function — and
show that it is determined by the asymptotic expansion of the volume.
Conversely, we
prove that this spectrum, together with its multiplicities and
coefficients, can be recovered recursively from the volume asymptotics by
an explicit reconstruction procedure.

We also give complementary interpretations in terms of arc spaces,
where the divisorial exponents appear both as ratios of vanishing orders
along generic divisorial arcs and as normalized codimension growth rates
of divisorial cylinders.

Taken together, these results establish an explicit correspondence
between the visible spectrum and the asymptotic expansion of
sublevel-set volumes, providing an intrinsic metric characterization of
the visible spectrum itself.
\end{abstract}

\maketitle

\section{Introduction}
\label{sec:intro}

The log canonical threshold (LCT) is a fundamental invariant in the
study of singularities, measuring the complexity of a germ through a
single numerical value. Its real-analytic analogue, the real log
canonical threshold (RLCT) of Watanabe, is defined through resolutions
over $\R$ rather than over $\C$ and plays an equally important role in
singular learning theory. As shown by Saito~\cite[Corollaries 1
and~2]{Saito21}, the two invariants genuinely need not agree in general,
since a real log resolution need not induce a complex log resolution of
the complexified singularity (and vice versa). Throughout the present
paper we work with complex log resolutions in the sense of Hironaka, so
the invariant studied here is the (complex) LCT, not Watanabe's RLCT; see
the remark following Definition~\ref{def:spectrum} for the precise
normalization used.

The present paper sits at the interface of this birational-geometric
circle of ideas and asymptotic analysis: the divisorial data of a log
resolution is what \emph{produces} the asymptotic expansion of the
sublevel-set volume, and our main results concern what can, and cannot,
be recovered from that expansion alone, without reference to a
resolution. The geometric input is thus a means to an analytic end
rather than the object ultimately reconstructed.

The RLCT plays an analogous role in singular learning theory, governing
the asymptotic behaviour of statistical models with singular parameter
spaces~\cite{Watanabe09,Wei23}; a related metric viewpoint relates
Łojasiewicz-type invariants to log canonical thresholds and their
bi-Lipschitz stability~\cite{BiviAFukui18}. These connections are
mentioned only for context: the present work concerns exclusively the
complex LCT and the asymptotics arising from complex log resolutions.

Motivated by these developments, we study analytic singularities through
the asymptotic behaviour of sublevel-set volumes. Given an analytic
function $f$, we consider
\[
V(\varepsilon)
=
\Vol\{x:|f(x)|<\varepsilon\},
\]
which measures the rate at which the singularity collapses at small
scales.

It is classical that the leading decay of $V(\varepsilon)$ is governed
by the LCT. More generally, the theory of local zeta functions shows
that $V(\varepsilon)$ admits a complete asymptotic expansion whose
exponents are precisely the poles of the associated local zeta
function. The natural inverse problem is then the following:

\begin{quote}
\emph{To what extent can the divisorial information encoded in the poles
of the local zeta function be recovered directly from the asymptotic
behaviour of the volume function?}
\end{quote}

More generally, let $(X,0)$ be the germ of a complex analytic space of
pure dimension $n$, and let
\[
I=(f)\subset\mathcal O_{X,0}
\]
be a principal ideal, generated by a single holomorphic germ
$f:(X,0)\to(\C,0)$. Throughout this paper we work exclusively with
principal analytic ideals; the extension to arbitrary analytic ideals
will be addressed elsewhere.

The main purpose of this paper is to study the relationship between the
asymptotic expansion of the sublevel-set volume and the \emph{visible
spectrum} — the set of poles of the local zeta function that occur with
nonzero Laurent coefficient in the expansion of $V(\varepsilon)$, a
locally finite birational invariant contained in the extended candidate
spectrum of any log resolution.

More precisely, we prove that:

\begin{itemize}

\item the asymptotic expansion of $V(\varepsilon)$ is governed by
the visible spectrum;

\item conversely, the visible spectrum, together with its
multiplicities and coefficients, is reconstructed recursively from
that expansion by an explicit intrinsic procedure.

\end{itemize}

\medskip
\noindent\textbf{Known results and new contributions.}
The existence of asymptotic expansions of sublevel-set volumes and their
relation to the poles of local zeta functions is classical, originating
in the work of Atiyah~\cite{Atiyah70},
Varchenko~\cite{Varchenko76}, and
Igusa~\cite{Igusa00},
and later refined by
Denef--Loeser~\cite{DenefLoeser92,DenefLoeser99}.
Likewise, the interpretation of divisorial ratios in terms of arc spaces
follows from the work of
Ein--Lazarsfeld--Mustaţă~\cite{EinLazMust04}
and
Mustaţă~\cite{Mustata01},
and is included here to place the reconstruction theorem in its natural
birational context.

The principal contribution of this article is the converse direction
described above. We also clarify the distinction between the
extended candidate spectrum
$\widetilde{\Lambda}_\pi(I)$ and the visible spectrum
$\Lambda(I)$ of actual poles of the local zeta function, making
explicit the hypotheses under which the reconstruction is valid.

\section{Preliminaries}
\label{sec:prelim}

Throughout the paper we work in the analytic category. The basic object of
study is an analytic pair
\[
(X,I),
\]
where $X$ is a complex analytic space of pure dimension $n$ and
$I=(f)\subset\mathcal O_X$ is a principal analytic ideal, generated by a
single holomorphic germ $f$.

Our interest lies in the asymptotic geometry determined by the ideal.
Rather than viewing $I$ merely as defining a singular subset, we regard
it as encoding a hierarchy of vanishing scales whose birational
resolution governs both the local zeta function and the asymptotic
behaviour of sublevel-set volumes.

Throughout the remainder of the paper we work locally in a smooth ambient
space,
\[
(X,0)=(\C^n,0),
\]
and all constructions are understood in a sufficiently small
neighbourhood of the origin. This is the natural setting for the study
of local zeta functions and volume asymptotics. Although the ambient
space is complex, all metric and volumetric quantities are taken with
respect to the underlying real structure
\[
\C^n\cong\R^{2n}
\]
and the associated Lebesgue measure.

The general analytic pair $(X,I)$ is introduced only to motivate the
setting. From this point onward the singularity under consideration is
that of the ideal $I=(f)$, not of the ambient space itself. A different
asymptotic theory based on the same divisorial framework, but formulated
over arbitrary reduced analytic spaces, is developed independently
in~\cite{DAH}.

Let
\[
I=(f)\subset\mathcal O_{X,0}
\]
be a principal ideal, generated by a holomorphic germ $f$ that is not a
unit. By coherence, all invariants considered below
depend only on the germ.

The associated analytic set is
\[
V(I)
=
\{x\in X:\ f(x)=0\}.
\]

Fix once and for all a sufficiently small Euclidean ball
\[
B_\delta\subset\C^n.
\]
Unless explicitly stated otherwise, every asymptotic statement is taken
as $\varepsilon\to0^+$ with $\delta$ fixed.

\subsection*{Resolution and birational data}

Resolution of singularities provides the fundamental bridge between the
analytic and birational aspects of the theory. Rather than resolving only
the singularities of the ambient space, one resolves the vanishing
behaviour of the ideal itself.

A \emph{log resolution} of the pair $(X,I)$ is a proper birational
morphism
\[
\pi:Y\longrightarrow X
\]
such that

\begin{itemize}
\item $Y$ is smooth, and
\item the divisor
\[
\Exc(\pi)\cup\Supp(\pi^{-1}I)
\]
has simple normal crossings.
\end{itemize}

Such resolutions exist in characteristic zero by Hironaka's theorem
\cite{Hir64}; see also the functorial analytic desingularization of
Bierstone--Milman~\cite{BM97}.

Writing the total transform of the ideal and the relative canonical
divisor as
\[
\pi^*I
=
\Oh_Y\!\left(-\sum_i m_iE_i\right),
\qquad
K_{Y/X}
=
\sum_i k_iE_i,
\]
the integers
\[
m_i=\nu_{E_i}(I),
\qquad
k_i=a(E_i),
\]
are respectively the multiplicity of $I$ and the discrepancy along the
prime divisor $E_i$. Here $E_i$ ranges over \emph{all} prime components of
the simple normal crossings configuration
$D_\pi:=\Exc(\pi)\cup\Supp(\pi^{-1}I)$, not only over the exceptional ones:
whenever $\pi$ requires no exceptional divisor over a given component (for
instance when $I$ is already a normal crossings divisor, so that
$\pi=\mathrm{id}$ is itself a log resolution), the corresponding
strict-transform components of $\pi^{-1}I$ still carry well-defined data
$(m_i,k_i)$, with $k_i=0$ when $E_i$ is non-exceptional. Throughout this
paper we use the classical notation $(m_E,k_E)$. Readers familiar with the
homological theory
developed in~\cite{DAH} should note that the equivalent notation
$(\nu_E,a_E)$ is adopted there; the two pairs of symbols denote exactly
the same numerical invariants.

Recall that $I=(f)$ is principal, generated by the single germ $f$. In
local coordinates adapted to
the exceptional divisor one has
\[
f\circ\pi
=
u(y)\prod_i y_i^{m_i},
\qquad
|\Jac(\pi)|
=
v(y)\prod_i|y_i|^{k_i},
\]
where $u$ and $v$ are analytic units.

The associated divisorial ratios are
\[
\lambda_E
=
\frac{k_E+1}{m_E},
\]
which simultaneously govern the integrability of $|f|^{-2s}$, the poles
of the local zeta function, the asymptotic expansion of sublevel-set
volumes, and the arc-space interpretation recalled later in
Section~\ref{sec:arc}.

\subsection*{LCT and the resolution-dependent spectrum}

The \emph{log canonical threshold} of the ideal $I$ is
\[
\rlct(I)
=
\min_E\frac{k_E+1}{m_E}.
\]

Equivalently, it is the largest exponent $\lambda$ for which
$|f|^{-2\lambda}$ is locally integrable near the origin, where $f$
generates $I$.

\begin{definition}[Resolution-dependent spectrum]
\label{def:spectrum}

For a fixed log resolution $\pi:Y\to X$, the
\emph{resolution-dependent spectrum} is the finite set
\[
\Lambda_\pi(I)
=
\left\{
\frac{k_E+1}{m_E}
:
E\subset D_\pi
\right\},
\qquad
D_\pi:=\Exc(\pi)\cup\Supp(\pi^{-1}I),
\]
where $E$ ranges over \emph{all} prime divisors of the simple normal
crossings configuration $D_\pi$, including (when present) the
non-exceptional components coming from the strict transform of $V(I)$; see
Examples~\ref{ex:fermat} and~\ref{ex:log}, where such non-exceptional
components contribute the value $\lambda=1$ to $\Lambda_\pi(I)$.

Writing
\[
\Lambda_\pi(I)
=
\{\lambda_1<\cdots<\lambda_N\},
\]
one has
\[
\lambda_1
=
\rlct(I).
\]

\end{definition}

The spectrum $\Lambda_\pi(I)$ depends on the chosen resolution.
Additional blow-ups may introduce new exceptional divisors carrying
strictly larger ratios, while the minimum remains unchanged.
Consequently, the individual ratios attached to a fixed divisor are
birational invariants of the corresponding valuation, whereas the finite
set $\Lambda_\pi(I)$ itself is generally not.

We now introduce the central object studied in this paper: the
\emph{visible spectrum}, a notion proposed in the present work, defined
analytically through the local zeta function and, equivalently, through
the asymptotic expansion of sublevel-set volumes.

\begin{definition}[Visible spectrum]
\label{def:visible-spectrum}

Let $I=(f)$ be principal, and let

\[
Z_I(s)
=
\int_{B_\delta}|f(x)|^{-2s}\,d\mathrm{Vol}(x)
\]

be its local zeta function (cf.\ the \emph{Volume asymptotics and local
zeta function} subsection below, where the same integral is related to
$V(\varepsilon)$). The \emph{visible spectrum}

\[
\Lambda(I)
\]

is the set of actual poles of $Z_I(s)$.

Equivalently, it is the set of exponents appearing with nonzero
coefficient in the asymptotic expansion of the sublevel-set volume
$V(\varepsilon)$.
The equivalence follows from the classical Mellin--Tauberian
correspondence between local zeta functions and volume asymptotics.

\end{definition}

By construction, $\Lambda(I)$ is defined directly as the pole set of the
integral $Z_I(s)$, with no reference to any particular resolution:
unlike $\Lambda_\pi(I)$, which is extracted from a chosen log
resolution, $\Lambda(I)$ admits an intrinsic characterization directly
from the asymptotic behaviour of the volume function. Its birational
invariance and the identity
\[
\min\Lambda(I)
=
\rlct(I)
\]
are established in Proposition~\ref{prop:resolution-independence}(iv)
below (see also \cite{Igusa00,DenefLoeser92,DenefLoeser99} for the
classical theory underlying this invariance). In general $\Lambda(I)$ is
only locally finite, since it is contained in the extended candidate
spectrum introduced in Definition~\ref{def:spectrum-extended}.

The valuative nature of the divisorial ratios is reflected in the
arc-space description of log canonical thresholds. By the work of
Mustaţă~\cite{Mustata01,Mustata06}, the quantities

\[
\frac{k_E+1}{m_E}
\]

appear as asymptotic codimension growth rates of divisorial cylinders in
jet schemes.

\medskip
Throughout this paper the log canonical threshold is defined, as is
standard in complex algebraic and analytic geometry (cf.\ \cite{Kollar13,
Lazarsfeld04}), through the integrability of
\[
|f|^{-2\lambda}
\]
with respect to a complex log resolution. If one instead adopts the
convention based on
\[
|f|^{-\lambda},
\]
which is common in singular learning theory following Watanabe's real
log canonical threshold (RLCT), the corresponding threshold is
$\frac12\rlct(f)$; this normalization discrepancy, together with the
distinct real-versus-complex resolution used to define it, is the reason
we avoid the notation RLCT for the invariant studied in this paper. The
same $\tfrac12$ normalization is used in the homological
theory developed in~\cite{DAH}, where the divisorial exponents are written

\[
\gamma_E
=
\frac12\lambda_E.
\]

Accordingly,

\[
\min\Gamma_H
=
\frac12\rlct(I),
\]

although this numerical correspondence plays no role in the present
paper.

\subsection*{Volume asymptotics and local zeta function}

Recall that $f$ generates the principal ideal
\[
I=(f).
\]
For $\varepsilon>0$ define the sublevel-set volume

\[
V(\varepsilon)
=
\Vol
\bigl\{
x\in B_\delta:\ |f(x)|<\varepsilon
\bigr\},
\]

where the volume is computed with respect to the
$2n$-dimensional Lebesgue measure
$d\mathrm{Vol}(x)$ on the underlying real manifold

\[
\C^n\simeq\R^{2n}.
\]

We write $d\mathrm{Vol}(x)$ (rather than $dV(x)$) to distinguish the
ambient measure from the one-variable volume function
$V(\varepsilon)$.

The associated local zeta function is

\[
Z_I(s)
=
\int_{B_\delta}
|f(x)|^{-2s}\,
d\mathrm{Vol}(x),
\]

which converges for

\[
0<\Re(s)<\lct(I).
\]

To relate the local zeta function to the push-forward measure associated
with the sublevel sets, consider the map

\[
\Phi:B_\delta\longrightarrow[0,\infty),
\qquad
\Phi(x)=|f(x)|.
\]

The push-forward of the $2n$-dimensional Lebesgue measure under $\Phi$
defines a positive Borel measure

\[
\mu=\Phi_*(d\mathrm{Vol})
\]

on $[0,\infty)$.

Its distribution function is precisely the sublevel-set volume,

\[
V(t)
=
\mu([0,t)).
\]

Equivalently, we denote by

\[
dV:=d\mu
\]

the associated Stieltjes measure.

The change-of-variables formula for push-forward measures gives

\[
Z_I(s)
=
\int_0^\infty
t^{-2s}\,
dV(t),
\]

which is the Mellin--Stieltjes transform of the measure $dV$.

Since $f$ is continuous on the compact ball $B_\delta$, there exists

\[
M
=
\sup_{x\in B_\delta}|f(x)|
<\infty.
\]

Hence

\[
V(t)
=
\Vol(B_\delta),
\qquad
t\ge M,
\]

so that $V$ is eventually constant and the Stieltjes measure $dV$
is supported in $[0,M]$.

Therefore,

\[
\Bigl[t^{-2s}V(t)\Bigr]_0^\infty
=
0,
\]

because

\[
V(0)=0,
\qquad
V(t)=\Vol(B_\delta)
\ \text{for } t\ge M,
\]

and

\[
\lim_{t\to\infty}t^{-2s}=0
\]

whenever $\Re(s)>0$.

Consequently, for

\[
0<\Re(s)<\lct(I),
\]

the Stieltjes integration-by-parts formula yields the equivalent
representation

\[
Z_I(s)
=
2s
\int_0^\infty
t^{-2s-1}
V(t)\,dt.
\]

Thus the local zeta function admits two equivalent descriptions:
first as the Mellin--Stieltjes transform of the measure $dV$, and,
after Stieltjes integration by parts, as an integral involving the
cumulative volume function $V$.

The meromorphic continuation of $Z_I(s)$ is a classical consequence of
resolution of singularities
\cite{Atiyah70,Igusa00,Varchenko76,DenefLoeser92}.
Indeed, after pulling back the integral to a log resolution, the
integrand becomes locally monomial, and the resulting monomial
integrals admit meromorphic continuation to the whole complex plane.

Classical Mellin inversion together with the standard asymptotic theory
of Mellin--Stieltjes transforms relates the poles of the local zeta
function, via this meromorphic continuation, to the
asymptotic expansion of the sublevel-set volume
\cite{Atiyah70,Igusa00,Varchenko76,DenefLoeser92,Watanabe09,Watanabe2024}.

From the perspective of singular learning theory
\cite{Watanabe09,Watanabe2024},
the function $V(\varepsilon)$ may be viewed as the cumulative
distribution of the push-forward measure induced by
$\Phi(x)=|f(x)|$.
The Mellin--Stieltjes representation above provides the natural bridge
between the geometry of sublevel sets and the analytic properties of
the local zeta function.

Throughout this paper we regard these analytic results as part of the
classical theory of local zeta functions. Our contribution begins from
this point, as described in the Introduction: reconstructing the visible
spectrum from the asymptotic behaviour of $V(\varepsilon)$ by means of
the divisorial data of a log resolution.

\section{Arc-Space Characterization of the Spectrum}
\label{sec:arc}

The purpose of this section is to provide two complementary intrinsic
descriptions of the same numerical invariants: one arising from the
geometric behaviour of analytic arcs approaching an exceptional divisor,
and the other from the asymptotic codimension growth of divisorial
cylinders in the arc space. Together, these descriptions yield the
arc-space characterisation stated in
Theorem~\ref{thm:arc-spectrum} and complete the
Volumetric--Divisorial Description.

We work locally at the origin $(X,0)=(\C^n,0)$.
By Denef--Loeser~\cite[§2.1]{DenefLoeser99}, the arc space carries a
natural structure as a projective limit of jet schemes and admits a
stratification by order of contact with divisors.

\begin{definition}[Arc space and jet schemes]\label{def:arc-space}
Let $X$ be a complex algebraic or analytic variety. For each $m\ge 0$,
the $m$-th jet scheme is
\[
  \mathcal L_m(X)
  :=
  \Hom\!\bigl(\Spec \C[t]/(t^{m+1}),\,X\bigr),
\]
and the \emph{arc space} is the projective limit
\[
  \mathcal L(X)
  :=
  \varprojlim_m \mathcal L_m(X).
\]
Its $\C$-points are in natural bijection with formal arcs
\[
  \gamma:\Spec \C[[t]]\to X.
\]
\end{definition}

\medskip

\begin{definition}[Cylinders and codimension]
A subset $C\subset\mathcal L(X)$ is a \emph{cylinder} if there exist
$m\ge0$ and a constructible subset $S\subset \mathcal L_m(X)$ such that
\[
  C=\psi_m^{-1}(S),
\]
where $\psi_m:\mathcal L(X)\to\mathcal L_m(X)$ is the truncation map.

Its \emph{codimension} is defined by
\[
  \codim_{\mathcal L(X)}(C)
  :=
  \codim_{\mathcal L_m(X)}(S)
  =
  (m+1)\dim X - \dim S.
\]
This quantity is independent of the choice of $m$ used to represent $C$.
Indeed, if $C = \psi_m^{-1}(S) = \psi_{m'}^{-1}(S')$ with $m' \ge m$, then
$S' = \tau_{m',m}^{-1}(S)$, where the truncation morphism
$\tau_{m',m} : \mathcal L_{m'}(X) \to \mathcal L_m(X)$ is a locally trivial
affine bundle of relative dimension $(m'-m)\dim X$. In particular,
\[
  \dim S' = \dim S + (m'-m)\dim X,
\]
and therefore
\[
  (m'+1)\dim X - \dim S'
  =
  (m+1)\dim X - \dim S.
\]
Thus the codimension is well defined.

Moreover, by \cite[§1, Lemma~1.8 and Cor.~1.9]{EinLazMust04}, this notion
coincides with the intrinsic codimension in $\mathcal L(X)$ endowed with
its Zariski topology.
\end{definition}

\medskip

For $g(t)=\sum a_k t^k\in\C[[t]]$, define
\[
  \ord_t(g)=\min\{k:a_k\neq0\},\qquad \ord_t(0)=+\infty.
\]
For an analytic function $h$ and an arc $\gamma$, set
\[
  \ord_\gamma(h):=\ord_t(h\circ\gamma).
\]

For the Euclidean radius $r(x)=\|x\|$,
\[
  \ord_\gamma(r)
  =
  \ord_t\|\gamma(t)\|
  =
  \min_i \ord_t(\gamma_i(t)).
\]

For an ideal $I\subset\mathcal O_{X,0}$, define
\[
  \ord_t(\gamma^*I)
  :=
  \min_{h\in I}\ord_t(h\circ\gamma),
\]
so that $\gamma^*I=(t^{\ord_t(\gamma^*I)})$.

\medskip

\begin{definition}[Contact loci]
For $m\ge0$, the \emph{contact locus} of order $m$ is
\[
  \Cont^m(I)
  :=
  \{\gamma\in\mathcal L(X):\ord_t(\gamma^*I)=m\}.
\]
\end{definition}

These sets are locally closed cylinders. Indeed, for every $r \ge m$ one has
\[
  \Cont^m(I) = \psi_r^{-1}(\Cont^m(I)_r),
\]
where $\Cont^m(I)_r \subset \mathcal L_r(X)$ is locally closed
(see \cite[Introduction]{EinLazMust04}).
They play a central role in the valuative description of singularities
(see \cite[Thm.~A]{EinLazMust04}, \cite[Prop.~4.1]{Mustata01}).

\medskip

Let $\pi:Y\to X$ be a log resolution with exceptional divisor
$E=\sum_i E_i$.

\begin{definition}[Divisorial cylinders]
Fix a log resolution $\pi:Y \to X$, and let $E_i \subset Y$ be a prime divisor.
For $k \ge 1$, define the \emph{divisorial cylinder associated to $E_i$ of order $k$} by
\[
  \Cont^k(E_i)
  :=
  \pi_\infty\bigl(\Cont^k(E_i;Y)\bigr),
\]
where
\[
  \Cont^k(E_i;Y)
  :=
  \{\gamma \in \mathcal L(Y) : \ord_\gamma(E_i)=k\}.
\]
\end{definition}

The sets $\Cont^k(E_i;Y)$ are locally closed cylinders in $\mathcal L(Y)$, and their images
$\Cont^k(E_i) \subset \mathcal L(X)$ are constructible cylinders.
If non-empty, they are irreducible, and their codimension is given by
\[
  \codim_{\mathcal L(X)} \Cont^k(E_i)
  =
  k\,(k_i+1),
\]
where $k_i$ is the coefficient of $E_i$ in the relative canonical divisor
$K_{Y/X} = \sum_i k_i E_i$
(see \cite[Theorem~A]{EinLazMust04}).

Moreover, the contact loci and divisorial cylinders are compatible:
for $\gamma\in\Cont^k(E)$,
\begin{equation}
\label{eq:contact-compat}
  \ord_t(\gamma^*I)
  =
  k\,\nu_E(I),
\end{equation}
where $\nu_E(I)=\min_{h\in I}\nu_E(h)$.

\medskip

\begin{definition}[Divisorial arc family]
\label{def:arc-family}
Let $\pi:Y\to X$ be a log resolution, and let $E\subset\Exc(\pi)$ be a prime divisor with smooth stratum $E^\circ$.
The associated \emph{arc family} is
\[
  \mathscr{A}_E
  :=
  \pi_\infty\Bigl(
  \bigl\{
  \gamma'\in\mathcal L(Y):
  \gamma'(0)\in E^\circ,\;
  \ord_{\gamma'}(E)=1
  \bigr\}
  \Bigr).
\]
\end{definition}

By the theory of Ein--Lazarsfeld--Mustaţă, the multi-contact loci
\[
  \Cont^\nu(E)
  :=
  \{\gamma'\in\mathcal L(Y):\ord_{\gamma'}(E_i)=\nu_i\}
\]
are irreducible locally closed cylinders in $\mathcal L(Y)$, and their images
under $\pi_\infty$ are constructible cylinders in $\mathcal L(X)$.
Moreover, these images correspond bijectively to the irreducible components
of contact loci (see \cite[Theorem~A]{EinLazMust04}).

In particular, the family $\mathscr{A}_E$ corresponds to the case
$\nu_i = 1$ for $E_i=E$ and $\nu_j=0$ for $j\ne i$.

In particular, every irreducible cylinder not dominating $X$
defines a divisorial valuation, and conversely every such valuation
arises from a divisor in a resolution
(cf.\ \cite[Theorem~C]{EinLazMust04}; see also
de~Fernex--Ein--Ishii~\cite{deFEI08}).

The family $\mathscr{A}_E$ is precisely the cylinder corresponding to
the valuation $\ord_E$, normalized by unit contact order.

\medskip

The following theorem describes the geometric structure of these arc families.

\begin{theorem}[Structure of divisorial arc families]
\label{thm:arc-structure}
Let $\pi:Y\to X$ be a log resolution of an ideal $I\subset\mathcal{O}_X$, and write
\[
I\cdot\mathcal{O}_Y=\mathcal{O}_Y\!\Bigl(-\sum_i m_i E_i\Bigr),
\qquad
K_{Y/X}=\sum_i k_i E_i.
\]
Fix a divisor $E=E_j$, and denote $m_E:=m_j$, $k_E:=k_j$, and let $\nu_E$ be the associated divisorial valuation.

Define
\[
\mathscr{A}_E
:=
\pi_\infty\bigl(\{\gamma'\in\mathcal L(Y):\ord_{\gamma'}(E)=1\}\bigr).
\]

Then:
\begin{enumerate}[label=\textnormal{(\roman*)}]
  \item $\mathscr{A}_E$ is a non-empty constructible cylinder in $\mathcal L(X)$.

  \item For a general arc $\gamma\in\mathscr{A}_E$, one has
  \[
  \ord_\gamma(I)=\nu_E(I)=m_E.
  \]

  \item There exists a dense subset $\mathscr{A}_E^{\mathrm{gen}}\subset\mathscr{A}_E$ such that for any $f\in I$ general,
  \[
  f(\gamma(t)) = u(t)\,t^{m_E}, \qquad u(t)\in\mathbb{C}[[t]]^\times.
  \]

  \item For $m\ge1$, the sets
  \[
  \pi_\infty\bigl(\{\gamma'\in\mathcal L(Y):\ord_{\gamma'}(E)=m\}\bigr)
  \]
  are constructible cylinders, and if non-empty they have codimension
  \[
  \codim_{\mathcal L(X)} = m\,(k_E+1).
  \]
\end{enumerate}
\end{theorem}

\begin{remark}
The codimension formula in (iv) follows from the description of contact loci via log resolutions. 
In fact, by \cite[Theorem~A]{EinLazMust04}, the image of a multi-contact locus has codimension
\[
\sum_i \nu_i (k_i+1).
\]
In the divisorial case, this gives $m(k_E+1)$.
\end{remark}

\medskip

We now derive from this structure the arc-space description of the resolution-dependent spectrum.

\begin{theorem}[Arc-space description of the spectrum]
\label{thm:arc-spectrum}

Let $\pi:Y\to X$ be a log resolution of an ideal $I\subset\mathcal{O}_X$, and write
\[
I\cdot\mathcal{O}_Y=\mathcal{O}_Y\!\Bigl(-\sum_i m_i E_i\Bigr),
\qquad
K_{Y/X}=\sum_i k_i E_i,
\]
with $E_i$ prime divisors. For each exceptional divisor $E=E_j$, set
$m_E:=m_j$, $k_E:=k_j$, and denote by $\nu_E$ the associated divisorial valuation.

Define the divisorial arc family
\[
\mathscr{A}_E
:=
\pi_\infty\bigl(\{\gamma'\in\mathscr{L}(Y):\ord_{\gamma'}(E)=1\}\bigr),
\]
and let $\mathscr{A}_E^{\mathrm{gen}}$ be the subset of arcs whose lift meets $E$
transversely and avoids all other exceptional divisors.

Then:
\begin{enumerate}[label=\textnormal{(\roman*)}]
  \item For every $\gamma\in\mathscr{A}_E^{\mathrm{gen}}$, one has
  \[
  \ord_\gamma(f)=m_E,
  \qquad
  \ord_{\tilde\gamma}(\Jac\pi)=k_E,
  \]
  and hence
  \[
  \lambda_E=\frac{k_E+1}{m_E}.
  \]

  \item For every integer $m\ge1$, the cylinder
  \[
  \Cont^m(E)
  :=
  \pi_\infty\bigl(\{\gamma'\in\mathscr{L}(Y):\ord_{\gamma'}(E)=m\}\bigr)
  \]
  is constructible and, if non-empty, satisfies
  \[
  \frac{\codim_{\mathscr{L}(X)}\Cont^m(E)}{m\,\nu_E(I)}
  =
  \frac{k_E+1}{m_E}.
  \]
\end{enumerate}

Consequently,
\[
  \Lambda_\pi(I)
  =
  \Bigl\{\frac{k_E+1}{m_E} : E\subset D_\pi\Bigr\}
  =
  \Bigl\{\frac{\codim\,\Cont^m(E)}{m\,\nu_E(I)} : E\subset D_\pi\Bigr\}.
\]
\end{theorem}

\begin{proof}
Fix a divisor $E=E_j$ and let $p\in E^\circ$. Choose local coordinates
$(y_1,\ldots,y_n)$ on $Y$ centered at $p$ such that
\[
E=\{y_1=0\}, \qquad
I\cdot\mathcal{O}_Y = \mathcal{O}_Y(-m_E E),
\]
and the total transform has simple normal crossings. Since $p\in E^\circ$,
no other exceptional divisor passes through $p$, and locally
\[
f\circ\pi(y) = u(y)\,y_1^{m_E},
\]
with $u$ a unit.

\medskip
\noindent\textbf{Proof of (i).}

Let $\gamma\in\mathscr{A}_E^{\mathrm{gen}}$ and let $\tilde\gamma$ be its lift to $Y$.

By definition of $\mathscr{A}_E^{\mathrm{gen}}$, the lifted arc meets $E$
transversely and avoids all other exceptional divisors, hence
\[
\ord_{\tilde\gamma}(E)=1,
\qquad
\ord_{\tilde\gamma}(E_i)=0 \quad (i\neq j).
\]

It follows that
\[
\ord_\gamma(f)
=
\ord_t\bigl(f\circ\pi\circ\tilde\gamma\bigr)
=
\ord_t\bigl(u(\tilde\gamma(t))\,t^{m_E}\bigr)
=
m_E.
\]

On the other hand, by the expression of the relative canonical divisor
\[
K_{Y/X}=\sum_i k_i E_i,
\]
one has
\begin{equation}
\label{eq:jac-order}
\ord_{\tilde\delta}(\Jac\pi)
= \sum_i k_i\,\ord_{\tilde\delta}(E_i)
\end{equation}
for every arc $\tilde\delta$ on $Y$.

This follows from the identity
\[
K_{Y/X}=\sum_i k_i E_i,
\]
which implies
\[
\ord_{\tilde\delta}(\Jac \pi)=\sum_i k_i\,\ord_{\tilde\delta}(E_i)
\]
(see Ein--Lazarsfeld--Mustaţă~\cite[Theorem~A]{EinLazMust04}).

Applying this to $\tilde\gamma$, we obtain
\[
\ord_{\tilde\gamma}(\Jac\pi)=k_E.
\]

Therefore
\[
\bigl(\ord_\gamma(f),\,\ord_{\tilde\gamma}(\Jac\pi)\bigr)
=
(m_E,\,k_E),
\]
which yields
\[
\lambda_E=\frac{k_E+1}{m_E}.
\]

\medskip
\noindent\textbf{Proof of (ii).}

By the structure theorem for divisorial cylinders (cf.\ \cite[Theorem~A]{EinLazMust04}),
the set
\[
\{\gamma'\in\mathscr{L}(Y):\ord_{\gamma'}(E)=m\}
\]
is an irreducible locally closed cylinder of codimension $m$ in $\mathscr{L}(Y)$,
and its image under $\pi_\infty$ is a constructible cylinder in $\mathscr{L}(X)$
of codimension
\[
\codim_{\mathscr{L}(X)}\Cont^m(E)
=
m\,(k_E+1).
\]

Since $\nu_E(I)=m_E$, it follows that
\[
\frac{\codim_{\mathscr{L}(X)}\Cont^m(E)}{m\,\nu_E(I)}
=
\frac{m(k_E+1)}{m\,m_E}
=
\frac{k_E+1}{m_E}.
\]
\end{proof}

\begin{remark}
\label{rem:two-faces}
Theorem~\ref{thm:arc-spectrum} shows that each $\lambda_E\in\Lambda_\pi(I)$ admits two intrinsic descriptions:
\begin{enumerate}
  \item \emph{Valuative:} $\lambda_E$ is determined by the pair
        $(\ord_\gamma(f),\,\ord_{\tilde\gamma}(\Jac\pi))$
        along generic arcs approaching $E$.

  \item \emph{Geometric:} $\lambda_E$ is the normalized codimension growth rate of the
        divisorial cylinders $\Cont^m(E)$ in $\mathscr{L}(X)$.
\end{enumerate}
Their coincidence reflects the structure of arc spaces under resolution of singularities,
as described in \cite{EinLazMust04}, and provides the geometric foundation for the
Volumetric--Divisorial Description developed in this work.
\end{remark}

\subsection{Metric interpretation along divisorial arcs}
\label{subsec:metric}

Theorem~\ref{thm:arc-spectrum}(i) computes the order of vanishing of $f$
along a generic arc approaching $E$, namely $\ord_\gamma(f)=m_E$. To compare the asymptotic behaviour of $f$ with the ambient metric along
a generic divisorial arc, one must also determine the order of
vanishing of the ambient radius $r(x)=\|x\|$. This is governed not
by the discrepancy $k_E$, but by the value of the divisorial valuation on
the maximal ideal of the origin,
\[
\nu_E(\mathfrak m):=\min_{1\le j\le n}\nu_E(x_j),
\]
that is, by the order to which the pulled-back coordinate functions vanish
along $E$.

\begin{remark}[The radius is measured by $\nu_E(\mathfrak m)$, not by $k_E+1$]
\label{rem:radius-warning}
It is tempting to read the codimension formula
$\codim_{\mathscr L(X)}\Cont^m(E)=m(k_E+1)$ of
Theorem~\ref{thm:arc-spectrum}(ii) as an order of vanishing of the radius
along an arc. This is incorrect: $k_E+1=a(E)+1$ is the log discrepancy and
enters as an arc-space codimension (a count of jets), whereas
$\ord_\gamma(r)$ is the order of vanishing of a single analytic function,
the radius, along a single arc. In general
\[
\ord_\gamma(r)=\nu_E(\mathfrak m)\neq k_E+1.
\]
For instance, for the blow-up of the origin in $\mathbb C^n$ one has
$\nu_E(\mathfrak m)=1$ while $k_E+1=n$. The two quantities coincide only in
exceptional situations.
\end{remark}

Once the correct order is used, the resolution yields the pointwise scaling
\[
|f(\gamma(t))|\asymp r(\gamma(t))^{\,m_E/\nu_E(\mathfrak m)}
\]
along analytic curves whose lifts approach $E^\circ$ transversely. The
exponent $m_E/\nu_E(\mathfrak m)$ is in general different from
$1/\lambda_E=m_E/(k_E+1)$.

\begin{proposition}[Metric scaling along divisorial arcs]
\label{prop:metric-realisation}
Let $E\subset\Exc(\pi)$ and $p\in E^\circ$. Let
$\gamma:(0,\varepsilon_0)\to\mathbb{C}^n$ be an analytic arc whose
lift $\tilde\gamma$ to $Y$ satisfies
\[
\tilde\gamma(t)\to p, \qquad \ord_{\tilde\gamma}(E)=1,
\]
and avoids all other exceptional divisors. Then
\begin{equation}
\label{eq:metric-scaling}
  |f(\gamma(t))| \asymp r(\gamma(t))^{\,m_E/\nu_E(\mathfrak m)}
  \qquad(t\to0),
\end{equation}
where $\nu_E(\mathfrak m)=\min_j\nu_E(x_j)$ and the implicit constants
depend only on the resolution chart near $p$. Equivalently,
$\ord_\gamma(f)=m_E$ and $\ord_\gamma(r)=\nu_E(\mathfrak m)$.
\end{proposition}

\begin{proof}
Work in local coordinates $(y_1,\ldots,y_n)$ centred at a point
$p\in E^\circ$. After possibly shrinking the neighbourhood of $p$,
the morphism $\pi:Y\to\mathbb{C}^n$ has monomial form in these
coordinates, compatible with the normal crossings structure
(see \cite[§1]{EinLazMust04}). In particular,
\[
  E=\{y_1=0\}, \qquad
  f\circ\pi(y) = u(y)\,y_1^{m_E},
\]
with $u$ an analytic unit.

Since $\tilde\gamma(t)\to p\in E^\circ$ and
$\ord_{\tilde\gamma}(E)=1$, we have
\[
y_1(\tilde\gamma(t)) \to 0,
\qquad
y_i(\tilde\gamma(t)) \to c_i \neq 0 \quad (i\ge2),
\]
and $\ord_t\bigl(y_1(\tilde\gamma(t))\bigr)=1$. Thus $u(\tilde\gamma(t))$
remains bounded above and below by positive constants, and therefore
\begin{equation}
\label{eq:f-order}
  |f(\gamma(t))|
  = |u(\tilde\gamma(t))|\,|y_1(\tilde\gamma(t))|^{m_E}
  \asymp |t|^{m_E}.
\end{equation}

We now compute $\ord_\gamma(r)$. Since all norms on $\mathbb{C}^n$ are
equivalent,
\[
r(\gamma(t)) = \|\pi(\tilde\gamma(t))\|
\asymp \max_{1\le j\le n} |x_j(\pi(\tilde\gamma(t)))|,
\]
where $x_1,\ldots,x_n$ are the ambient coordinates. The pulled-back
function $x_j\circ\pi$ is analytic on $Y$ and, near a general point
$p\in E^\circ$ avoiding all other exceptional divisors and the strict
transform of $\{x_j=0\}$, its divisor is $\nu_E(x_j)\,E$. Hence in the
chosen chart
\[
x_j\circ\pi(y)=w_j(y)\,y_1^{\nu_E(x_j)},
\qquad w_j(p)\neq0,
\]
so that $\ord_{\tilde\gamma}(x_j\circ\pi)=\nu_E(x_j)$, using
$\ord_{\tilde\gamma}(E)=1$. Taking the maximum over $j$ gives
\begin{equation}
\label{eq:r-order}
  \ord_\gamma(r)
  =\min_{j}\ord_{\tilde\gamma}(x_j\circ\pi)
  =\min_{j}\nu_E(x_j)
  =\nu_E(\mathfrak m),
\end{equation}
and therefore $r(\gamma(t))\asymp|t|^{\nu_E(\mathfrak m)}$.

Combining \eqref{eq:f-order} and \eqref{eq:r-order} yields
\[
  |f(\gamma(t))|
  \asymp
  |t|^{m_E}
  =
  \bigl(|t|^{\nu_E(\mathfrak m)}\bigr)^{m_E/\nu_E(\mathfrak m)}
  \asymp
  r(\gamma(t))^{\,m_E/\nu_E(\mathfrak m)}.
\]
This completes the proof.
\end{proof}

\begin{remark}[Valuative versus metric interpretation]
\label{rem:valuative-metric}
Theorem~\ref{thm:arc-spectrum}(i) provides a valuative description of
$\lambda_E$ in terms of the pair $(\ord_\gamma(f),\ord_{\tilde\gamma}(\Jac\pi))
=(m_E,k_E)$ along formal arcs, while
Proposition~\ref{prop:metric-realisation} gives the analytic scaling along
curves,
\[
|f(\gamma(t))|\asymp r(\gamma(t))^{\,m_E/\nu_E(\mathfrak m)},
\qquad
\frac{\ord_\gamma(r)}{\ord_\gamma(f)}=\frac{\nu_E(\mathfrak m)}{m_E}.
\]
These two descriptions involve genuinely different data: the divisorial
ratio $\lambda_E=(k_E+1)/m_E$ is built from the discrepancy $k_E$ and
governs the volumetric exponent (Section~\ref{sec:vol-spectrum}), whereas
the pointwise scaling exponent $m_E/\nu_E(\mathfrak m)$ is built from the
order $\nu_E(\mathfrak m)$ of the maximal ideal and governs the decay of
$|f|$ relative to the ambient radius. They coincide precisely when
$\nu_E(\mathfrak m)=k_E+1$, which does not hold in general. The metric
content of $\lambda_E$ is therefore carried by the volume function
$V(\varepsilon)$, not by the arcwise comparison of $|f|$ with $r$.
\end{remark}
\section{Volumetric Asymptotics and Divisorial Spectrum}
\label{sec:vol-spectrum}

We retain the notation introduced in Section~2. In particular,
$\pi : Y \to \C^n$ denotes a log resolution of $I$ with exceptional
divisors $\{E_i\}$ and associated numerical data $(k_i, m_i)$.
The sublevel-set volume is denoted by $V(\varepsilon)$, and
$\lambda_1 = \rlct(I) = \min \Lambda(I)$.

The analytic behaviour of the local zeta function reflects
the divisorial structure of the singularity. In particular, its poles
are controlled by the numerical data of a log resolution; see
Varchenko~\cite{Varchenko76}, Igusa~\cite{Igusa00}, and
Denef--Loeser~\cite{DenefLoeser92}.

We now make this correspondence precise.

\begin{definition}[Extended candidate spectrum]
\label{def:spectrum-extended}
For a fixed log resolution $\pi:Y\to X$ of $I$, the \emph{extended
candidate spectrum} is
\[
\widetilde\Lambda_\pi(I)
\;:=\;
\bigcup_{E_i\subset D_\pi}
\left\{\frac{k_i+1+j}{m_i}\ :\ j=0,1,2,\dots\right\}
\;\supseteq\;
\Lambda_\pi(I),
\]
i.e.\ the ratios of Definition~\ref{def:spectrum} together with all their
positive integer translates $j/m_i$ along each divisor. This set is
\emph{locally finite} (finite intersection with any bounded interval), but
in general infinite.
\end{definition}

\begin{proposition}[Divisorial pole identification]
\label{prop:poles}
The local zeta function $Z_I(s)$ admits a meromorphic continuation to
$\C$, and for every log resolution $\pi$,
\[
\operatorname{Poles}(Z_I) \;\subseteq\; \widetilde\Lambda_\pi(I).
\]
Moreover the order of the pole at any $\lambda\in\widetilde\Lambda_\pi(I)$ is
bounded above by the number of exceptional divisors of $\pi$ that meet
simultaneously at a common point of a resolution chart.
\end{proposition}

\begin{proof}
This is Igusa's theorem on the poles of the archimedean complex power via
desingularization~\cite[Thm.~5.4.1]{Igusa00}, applied to a log resolution
$\pi$ of $(X,I)$ with numerical data $(N_E,n_E)=(m_i,k_i+1)$ in Igusa's
notation. For $K=\C$ and the untwisted quasicharacter (the case relevant
here, with $|p|=0$ in Igusa's statement), the poles of the complex power lie
in $-\tfrac{1}{N_E}(n_E+\N)$ for each $E$, in a variable related to our $s$
by a sign flip; translating this to our convention gives exactly the stated
containment in $\widetilde\Lambda_\pi(I)$.
The multiplicity bound is Igusa's statement that the order of any pole is at
most the dimension of the nerve complex of the normal-crossing
configuration $\Exc(\pi)\cup\operatorname{Supp}(\pi^{-1}I)$, increased by
$1$: a $p$-simplex of the nerve complex corresponds to $p+1$ divisors with
nonempty common intersection, so this is precisely the number of divisors
meeting at a point.
\end{proof}

Proposition~\ref{prop:poles} requires more than the leading divisorial ratios
$(k_i+1)/m_i$. The Taylor expansion of the smooth unit factor in the local
resolution integral
\[
\int_U\prod_{i\in I_U}|y_i|^{2(k_i-sm_i)}\Phi(y)\,dy
\]
produces the shifted candidates
\[
\frac{k_i+1+j}{m_i},\qquad j\ge1,
\]
which together with the leading ratios form the extended candidate spectrum
$\widetilde\Lambda_\pi(I)$. We are grateful to Guillem Blanco for pointing
out this phenomenon and for the example
$f=x^5+y^7+x^3y^3$, whose pullback
\[
f\circ\pi=\sigma^{35}\bigl[(1+t^7)+\sigma t^3\bigr]
\]
exhibits the additional candidate $13/35$ beyond the leading ratio $12/35$.

\begin{remark}[Candidate versus actual poles]
\label{rem:candidate-vs-actual}
The extended spectrum depends only on the resolution data, whereas the
visible spectrum
\[
\Lambda(I)\subseteq\widetilde\Lambda_\pi(I)
\]
consists of those candidates whose transversal contributions survive after
summation over all charts. Since cancellations are common
\cite[\S2.3]{Denef91}, $\widetilde\Lambda_\pi(I)$ is typically much larger
than $\Lambda(I)$; Example~\ref{ex:newton} illustrates this phenomenon.
\end{remark}

\begin{remark}[When the candidate spectrum is complete]
\label{rem:when-complete}
In favorable situations, however, no shifted candidates occur. This is the
case, for instance, for toric resolutions of Newton--non-degenerate
singularities whose monomials all lie on the boundary of the Newton
polyhedron, where
\[
\widetilde\Lambda_\pi(I)=\Lambda_\pi(I).
\]
Then the spectrum is finite, the multiplicity bound of
Proposition~\ref{prop:poles} reduces to the sharp chart-by-chart count
$m_\lambda=\#\{i:\lambda_i=\lambda\}$, and
Examples~\ref{ex:bp}, \ref{ex:fermat}, and \ref{ex:log} apply.
\end{remark}

Finally, the uniqueness of asymptotic expansions implies that the recursive
procedure below recovers the visible spectrum $\Lambda(I)$, together with
its multiplicities and coefficients, uniquely from the volume function
$V(\varepsilon)$. The whole reconstruction rests on a single elementary
comparison between the basis functions
$\varepsilon^{2\lambda}|\log\varepsilon|^{m-1}$ appearing in the expansion,
which we isolate as a lemma.

\begin{lemma}[Dominance]
\label{lem:dominance}
Let $\lambda_1<\lambda_2$ and $m_1,m_2\ge1$. Then
\[
\varepsilon^{2\lambda_2}|\log\varepsilon|^{m_2-1}
=
o\bigl(\varepsilon^{2\lambda_1}|\log\varepsilon|^{m_1-1}\bigr),
\qquad
\varepsilon\to0^+.
\]
\end{lemma}

\begin{proof}
The ratio of the two terms equals
$\varepsilon^{2(\lambda_2-\lambda_1)}|\log\varepsilon|^{m_2-m_1}$,
which tends to $0$ as $\varepsilon\to0^+$ since $\lambda_2>\lambda_1$ and any
power of $|\log\varepsilon|$ is dominated by a positive power of
$\varepsilon^{-1}$ as $\varepsilon\to0^+$.
\end{proof}

\begin{lemma}
\label{lem:logpower}
Let $L(\varepsilon)=|\log\varepsilon|$. Then, as $\varepsilon\to0^+$,
\[
L(\varepsilon)^\alpha
\longrightarrow
\begin{cases}
0, & \alpha<0,\\
1, & \alpha=0,\\
+\infty, & \alpha>0.
\end{cases}
\]
\end{lemma}

\begin{proof}
Immediate from $L(\varepsilon)\to+\infty$ as $\varepsilon\to0^+$.
\end{proof}

The reconstruction procedure below relies only on
Lemma~\ref{lem:dominance}: the asymptotic scales
$\varepsilon^{2\lambda}|\log\varepsilon|^{m-1}$ are totally ordered by the
lexicographic order on $(\lambda,m)$, so the term of lowest order
dominates at every stage of the recursion. In this sense
Theorem~\ref{thm:volume-reconstruction} is a uniqueness theorem for
logarithmic asymptotic expansions, applied to the particular expansion
produced by the volume function.

\begin{theorem}[Reconstruction of the visible spectrum from volume asymptotics]
\label{thm:volume-reconstruction}

Let $I=(f)\subset\mathcal O_n$ be a principal analytic ideal, and define
\[
V(\varepsilon)
=
\Vol\{x\in B_\delta:\ |f(x)|<\varepsilon\}.
\]

Assume that $V(\varepsilon)$ admits the classical asymptotic expansion
arising from the theory of local zeta functions and resolution of
singularities (see
\cite{Varchenko76,Igusa00,DenefLoeser92}),
namely
\[
V(\varepsilon)
\sim
\sum_{\lambda\in\Lambda(I)}
c_\lambda\,
\varepsilon^{2\lambda}
|\log\varepsilon|^{m_\lambda-1},
\qquad
\varepsilon\to0^+,
\]
where
\[
\Lambda(I)
=
\{\lambda_1<\lambda_2<\cdots\}
\subset
\widetilde\Lambda_\pi(I)
\]
is the visible spectrum,
$m_\lambda$ denotes the order of the pole of the local zeta function
$Z_I(s)$ at $\lambda$, and the expansion is valid to arbitrary order:
by the classical asymptotic expansion theorem for local zeta functions
(see Varchenko~\cite{Varchenko76}, Igusa~\cite{Igusa00}, and
Denef--Loeser~\cite{DenefLoeser92}), the expansion is asymptotic in the
Poincar\'e sense to arbitrary finite order, i.e.\ for every $N$ the
difference between $V(\varepsilon)$ and the partial sum over
$\lambda\le N$ is $o\bigl(\varepsilon^{2N}|\log\varepsilon|^{m_N}\bigr)$
as $\varepsilon\to0^+$.

Then:

\begin{enumerate}[label=\textnormal{(\roman*)}]

\item
\textup{Recovery of exponents.}
The leading exponent is recovered by
\[
\lambda_1
=
\frac12
\lim_{\varepsilon\to0^+}
\frac{\log V(\varepsilon)}
{\log\varepsilon},
\]
and, recursively,
\[
\lambda_k
=
\frac12
\lim_{\varepsilon\to0^+}
\frac{\log|R_k(\varepsilon)|}
{\log\varepsilon},
\]
where
\[
R_k(\varepsilon)
=
V(\varepsilon)
-
\sum_{j<k}
c_{\lambda_j}
\varepsilon^{2\lambda_j}
|\log\varepsilon|^{m_{\lambda_j}-1}.
\]

\item
\textup{Recovery of multiplicities.}
For each recovered exponent $\lambda_k$, the multiplicity
$m_{\lambda_k}$ is the unique integer $p\ge1$ for which
\[
\frac{
R_k(\varepsilon)
}{
\varepsilon^{2\lambda_k}
|\log\varepsilon|^{p-1}
}
\]
converges to a finite nonzero limit.

\item
\textup{Recovery of coefficients.}
The coefficient $c_{\lambda_k}$ is recovered simultaneously as this
finite nonzero limit,
\[
\frac{
R_k(\varepsilon)
}{
\varepsilon^{2\lambda_k}
|\log\varepsilon|^{m_{\lambda_k}-1}
}
\longrightarrow
c_{\lambda_k}
\in\mathbb R\setminus\{0\}.
\]

\end{enumerate}

Consequently, the visible spectrum $\{\lambda:c_\lambda\neq0\}$, together
with its multiplicities and coefficients, is uniquely determined by the
asymptotic expansion of $V(\varepsilon)$. Therefore the complete visible
spectral data
\[
\bigl(
\{\lambda:c_\lambda\neq0\},
(m_\lambda),
(c_\lambda)
\bigr)
\]
is uniquely determined by the asymptotic expansion of
$V(\varepsilon)$.

\end{theorem}

\begin{proof}

The existence of the asymptotic expansion is classical. It follows from
resolution of singularities, the meromorphic continuation of the local
zeta function, and Mellin inversion
(see \cite{Varchenko76,Igusa00,DenefLoeser92}).
It therefore remains only to prove the reconstruction procedure.

\medskip
\noindent
\textbf{Leading term.}

By Lemma~\ref{lem:dominance}, since
\[
\lambda_1<\lambda_2<\cdots,
\]
the first term of the asymptotic expansion dominates all the others.
Hence
\[
V(\varepsilon)
=
c_{\lambda_1}
\varepsilon^{2\lambda_1}
|\log\varepsilon|^{m_{\lambda_1}-1}
\bigl(1+\eta(\varepsilon)\bigr),
\]
where
\[
\eta(\varepsilon)\longrightarrow0.
\]

Taking logarithms,

\[
\log V(\varepsilon)
=
2\lambda_1\log\varepsilon
+
(m_{\lambda_1}-1)\log|\log\varepsilon|
+
O(1),
\]

since $\log c_{\lambda_1}$ is constant and
$\log(1+\eta)=o(1)$.

Because

\[
\frac{\log|\log\varepsilon|}
{\log\varepsilon}
\longrightarrow0,
\]

we obtain

\[
\lambda_1
=
\frac12
\lim_{\varepsilon\to0^+}
\frac{\log V(\varepsilon)}
{\log\varepsilon}.
\]

\medskip
\noindent
\textbf{Induction.}

Suppose that

\[
\lambda_1,\ldots,\lambda_{k-1},
\qquad
m_{\lambda_j},
\qquad
c_{\lambda_j}
\]

have already been recovered.

Subtract their contributions and define

\[
R_k(\varepsilon)
=
V(\varepsilon)
-
\sum_{j<k}
c_{\lambda_j}
\varepsilon^{2\lambda_j}
|\log\varepsilon|^{m_{\lambda_j}-1}.
\]

The asymptotic expansion gives

\[
R_k(\varepsilon)
=
c_{\lambda_k}
\varepsilon^{2\lambda_k}
|\log\varepsilon|^{m_{\lambda_k}-1}
(1+o(1)).
\]

Since $1+o(1)\to1$, there exists $\varepsilon_0>0$ such that
$|o(1)|<1/2$ for all $\varepsilon<\varepsilon_0$, and hence
$1+o(1)>1/2$ for such $\varepsilon$. As $c_{\lambda_k}\neq0$, it follows
that $R_k(\varepsilon)$ has the same sign as $c_{\lambda_k}$ for all
sufficiently small $\varepsilon$. Therefore

\[
\log|R_k(\varepsilon)|
=
2\lambda_k\log\varepsilon
+
(m_{\lambda_k}-1)\log|\log\varepsilon|
+
O(1),
\]

and the same argument as in Step~1 yields

\[
\lambda_k
=
\frac12
\lim_{\varepsilon\to0^+}
\frac{\log|R_k(\varepsilon)|}
{\log\varepsilon}.
\]

Thus every visible exponent is recovered recursively.

\medskip
\noindent
\textbf{Recovery of multiplicity.}

Fix one recovered exponent $\lambda_k$ and define

\[
Q_{k,p}(\varepsilon)
=
R_k(\varepsilon)
\varepsilon^{-2\lambda_k}
|\log\varepsilon|^{-(p-1)},
\qquad
p\ge1.
\]

Then

\[
Q_{k,p}(\varepsilon)
=
c_{\lambda_k}
|\log\varepsilon|^{m_{\lambda_k}-p}
(1+o(1)).
\]

By Lemma~\ref{lem:logpower}, applied with $\alpha=m_{\lambda_k}-p$,

\[
\begin{cases}
|Q_{k,p}(\varepsilon)|\to\infty,
&
p<m_{\lambda_k},
\\[1ex]
Q_{k,p}(\varepsilon)\to0,
&
p>m_{\lambda_k},
\\[1ex]
Q_{k,p}(\varepsilon)\to c_{\lambda_k}\neq0,
&
p=m_{\lambda_k}.
\end{cases}
\]

Hence $m_{\lambda_k}$ is uniquely characterized as the only integer for
which the normalized remainder $Q_{k,p}(\varepsilon)$ converges to a
finite nonzero limit.

\medskip
\noindent
\textbf{Recovery of coefficient.}

Once $m_{\lambda_k}$ is identified as above, the same limit is precisely
the coefficient:
\[
Q_{k,m_{\lambda_k}}(\varepsilon)\longrightarrow c_{\lambda_k}.
\]

Applying this procedure inductively reconstructs every visible exponent,
its multiplicity, and its coefficient. Consequently

\[
\bigl(
\{\lambda:c_\lambda\neq0\},
(m_\lambda),
(c_\lambda)
\bigr)
\]

is uniquely determined by the asymptotic expansion of
$V(\varepsilon)$.

\end{proof}

Under the classical theory of local zeta functions
(\cite{Atiyah70,Varchenko76,Igusa00,DenefLoeser92}),
the coefficients $c_\lambda$ are precisely the leading Laurent
coefficients of the poles of $Z_I(s)$. Equivalently, they are obtained
by the standard residue construction after resolution of singularities
and therefore admit resolution-theoretic descriptions as sums of local
residue contributions over the relevant strata of a log resolution.

\begin{proposition}[Basic properties of the divisorial spectra]
\label{prop:resolution-independence}
Let $\pi,\pi_1,\pi_2:Y\to\mathbb{C}^n$ be log resolutions of $I$. Then:
\begin{enumerate}[label=\textnormal{(\roman*)}]
\item \emph{(Finiteness.)} Each resolution-dependent spectrum
$\Lambda_\pi(I)$ is finite.

\item \emph{(Valuative invariance.)} For a prime divisor $E$ over
$\mathbb{C}^n$, the quantities $\nu_E(I)$, $a(E)$ and
$\lambda_E=(a(E)+1)/\nu_E(I)$ depend only on the divisorial valuation
$\ord_E$, not on the model in which $E$ appears.

\item \emph{(Invariance of the minimum.)}
$\min\Lambda_\pi(I)=\rlct(I)$, independently of $\pi$.

\item \emph{(Invariance of the visible spectrum.)} The visible spectrum
$\Lambda(I)$ of Definition~\ref{def:visible-spectrum} is a birational
invariant, and $\Lambda(I)\subseteq\widetilde\Lambda_\pi(I)$ for every $\pi$
(Proposition~\ref{prop:poles}).

\item \emph{(Local finiteness, not finiteness, of $\Lambda(I)$.)}
$\Lambda(I)$ is locally finite, being contained in the locally finite set
$\widetilde\Lambda_\pi(I)$; it is not asserted to be finite (Remark~\ref{rem:candidate-vs-actual}).
\end{enumerate}
By contrast, the resolution-dependent spectrum $\Lambda_\pi(I)$ is
\emph{not} a birational invariant: in general
$\Lambda_{\pi_1}(I)\neq\Lambda_{\pi_2}(I)$
(Remark~\ref{rem:not-invariant}).
\end{proposition}

\begin{proof}
(i) is immediate, since the simple normal crossings configuration
$D_\pi=\Exc(\pi)\cup\Supp(\pi^{-1}I)$ has finitely many prime components.

\medskip
\noindent\textbf{(ii) Valuative invariance.}
Dominate $\pi_1$ and $\pi_2$ by a common log resolution
$\pi:Y\to\mathbb{C}^n$, with morphisms $\mu_i:Y\to Y_i$.

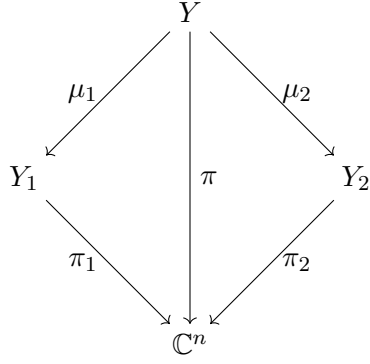
\begin{center}
\begin{tikzpicture}[scale=1.1]
\node (Y) at (0,2) {$Y$};
\node (Y1) at (-2,0) {$Y_1$};
\node (Y2) at (2,0) {$Y_2$};
\node (X) at (0,-2) {$\mathbb{C}^n$};

\draw[->] (Y) -- node[left] {$\mu_1$} (Y1);
\draw[->] (Y) -- node[right] {$\mu_2$} (Y2);
\draw[->] (Y1) -- node[left] {$\pi_1$} (X);
\draw[->] (Y2) -- node[right] {$\pi_2$} (X);
\draw[->] (Y) -- node[right] {$\pi$} (X);
\end{tikzpicture}
\end{center}

For a prime divisor $F\subset Y$ the valuation $\nu_F=\ord_F$, the value
\[
\nu_F(I):=\min_{h\in I}\nu_F(h),
\qquad
a(F):=\ord_F(K_{Y/\mathbb{C}^n}),
\]
depend only on $\nu_F$ and not on the model
(see \cite[§9]{KollarMori98}, \cite[§9.2]{Lazarsfeld04}, and
de~Fernex--Ein--Ishii~\cite{deFEI08}). If
$E_j^{(i)}\subset Y_i$ has strict transform $F$ in $Y$, then
\[
\frac{k_j^{(i)}+1}{m_j^{(i)}}
=
\frac{a(F)+1}{\nu_F(I)},
\]
which proves (ii). This says that each individual ratio is intrinsic to its
valuation; it does \emph{not} assert that the finite collection of ratios
extracted from a given resolution is independent of the resolution.

\medskip
\noindent\textbf{(iii) Invariance of the minimum.}
The resolution formula $\rlct(I)=\min_i (k_i+1)/m_i$ holds for every log
resolution, and a finer resolution only adds divisors whose ratios are
$\ge$ this minimum. Hence $\min\Lambda_\pi(I)=\rlct(I)$ for every $\pi$.

\medskip
\noindent\textbf{(iv) Invariance of the visible spectrum.}
By Proposition~\ref{prop:poles} the poles of $Z_I(s)$ lie in
$\widetilde\Lambda_\pi(I)$. The local zeta function $Z_I(s)$ — and therefore
its set of poles $\Lambda(I)$ — is an analytic invariant of $I$, independent
of any resolution (see \cite{Varchenko76,Igusa00}). Thus $\Lambda(I)$ is
birationally invariant and $\Lambda(I)\subseteq\widetilde\Lambda_\pi(I)$ for
every $\pi$.

\medskip
\noindent\textbf{(v) Local finiteness.}
Immediate from (iv) together with the local finiteness of
$\widetilde\Lambda_\pi(I)$ (Definition~\ref{def:spectrum-extended}); unlike
$\Lambda_\pi(I)$ itself, finer resolutions of the same $I$, or shifted
candidates along a single divisor, may in principle contribute infinitely
many further terms accumulating only at $+\infty$, so global finiteness of
$\Lambda(I)$ is not claimed.

The negative statement is exhibited in Remark~\ref{rem:not-invariant}.
\end{proof}

\begin{remark}[The resolution-dependent spectrum is not a birational invariant]
\label{rem:not-invariant}
Take $(X,I)=(\mathbb{C}^2,\mathfrak m)$, where $\mathfrak m=(x,y)$ is the
maximal ideal at the origin; here $\rlct(\mathfrak m)=2$. Let $\pi_1$ be the
blow-up of the origin, with single exceptional divisor $E_1$, for which
$a(E_1)=1$ and $\nu_{E_1}(\mathfrak m)=1$, so
\[
\Lambda_{\pi_1}(\mathfrak m)=\{2\}.
\]
Let $\pi_2$ be $\pi_1$ followed by the blow-up of a point $p\in E_1$,
introducing a second exceptional divisor $E_2$. In the local coordinates
$x=s,\ y=s^2t$ one finds $\ord_{E_2}(x)=1$, $\ord_{E_2}(y)=2$, hence
$\nu_{E_2}(\mathfrak m)=1$, while $a(E_2)=2$; therefore $\lambda_{E_2}=3$ and
\[
\Lambda_{\pi_2}(\mathfrak m)=\{2,3\}.
\]
Thus $\Lambda_{\pi_1}(\mathfrak m)\neq\Lambda_{\pi_2}(\mathfrak m)$: the
additional blow-up creates a new divisor carrying the strictly larger ratio
$3$. \emph{The resolution-dependent spectrum is generally not a birational
invariant.} Only the minimum
$\min\Lambda_{\pi}(\mathfrak m)=\rlct(\mathfrak m)=2$ and the visible
spectrum $\Lambda(\mathfrak m)$ — which retains only the value
$2$, the non-visible candidate $3$ being absent — are intrinsic. (In the
normalization $|f|^{-\lambda}$ used in singular learning theory, where the
thresholds are halved, these read $\{1\}$ and $\{1,\tfrac32\}$, with common
minimum $\tfrac12\rlct(\mathfrak m)=1$.)
\end{remark}

\section{Examples}
\label{sec:examples}

The following examples illustrate both directions of the
Volumetric--Divisorial Description. On the one hand, they show how the
divisorial data of a log resolution determines the asymptotic expansion
of the sublevel-set volume. On the other hand, they illustrate how the
visible spectrum can be recovered from that expansion by the
recursive procedure of Theorem~\ref{thm:volume-reconstruction}.

Throughout this section, $I=\langle f\rangle\subset\mathcal O_n$ denotes a
principal analytic ideal generated by a single function $f$, so that
\[
\rlct(f)=\rlct(I)
\]
trivially, with $f$ itself serving as the generator used in the resolution
formula below.

We recall that
\[
\rlct(I)=\min_i\frac{k_i+1}{m_i},
\]
where $k_i$ and $m_i$ denote respectively the discrepancy and the order
of vanishing of $I$ along the exceptional divisor $E_i$ of a fixed log
resolution
\[
\pi:Y\to\mathbb C^n.
\]

Unless otherwise stated, all volumes are computed with respect to the
$2n$-dimensional Lebesgue measure on
\[
\mathbb C^n\simeq\mathbb R^{2n}.
\]

\begin{example}[Brieskorn--Pham singularities]
\label{ex:bp}
Let $f(x,y)=x^a+y^b$ with integers $a,b\ge 2$, viewed as an
element of $\mathcal{O}_2=\mathcal{O}_{\mathbb{C}^2,0}$; the function is
non-degenerate with respect to its Newton polygon, since the face
polynomial $x^a+y^b$ has no critical points in $(\mathbb{C}^*)^2$.

\medskip

A log resolution of $\langle f\rangle$ at the origin can be obtained
via toric modifications adapted to the Newton polygon of $f$.
In dimension two, the Newton diagram has a single compact edge
joining $(a,0)$ to $(0,b)$, whose primitive inward normal is
\[
w=\frac{1}{\gcd(a,b)}(b,a).
\]

This Newton–polyhedral structure governs the candidate exponents and
the asymptotic behaviour of oscillatory integrals, and is compatible
with the resolution-based description of local zeta functions
(see Varchenko~\cite[§0.4--§0.6]{Varchenko76} and
Igusa~\cite[Ch.~5, §5.4]{Igusa00}).

The associated toric valuation $\nu_E$ satisfies
\[
  \nu_E(x^iy^j) = bi+aj,
\]
and since both monomials $x^a$ and $y^b$ lie on the Newton polygon,
\[
\nu_E(f)
=
\min\bigl(\nu_E(x^a),\nu_E(y^b)\bigr)
=
\min(ab,ab)
=
ab.
\]

\medskip

We compute the asymptotic behaviour of $V(\varepsilon)$ directly.
Apply the change of variables
\[
x=\varepsilon^{1/a}u,
\qquad
y=\varepsilon^{1/b}v
\]
in $\mathbb{C}^2\cong\mathbb{R}^4$. Then
\[
f(x,y)=\varepsilon\bigl(u^a+v^b\bigr),
\qquad
dx\wedge d\bar x\wedge dy\wedge d\bar y
=
\varepsilon^{2/a+2/b}\,
du\wedge d\bar u\wedge dv\wedge d\bar v.
\]
Therefore
\[
V(\varepsilon)
=
\varepsilon^{2(1/a+1/b)}
\Vol\bigl\{(u,v)\in\mathbb{C}^2:\ |u^a+v^b|<1\bigr\},
\]
which implies
\[
V(\varepsilon)\asymp \varepsilon^{2(1/a+1/b)}.
\]

Comparing this asymptotic behaviour with the general form
\[
V(\varepsilon)\asymp \varepsilon^{2\lambda},
\]
we conclude that
\[
\rlct(f)=\frac{1}{a}+\frac{1}{b},
\]
provided $\frac{1}{a}+\frac{1}{b}\le 1$.

In the case $\frac{1}{a}+\frac{1}{b}>1$, the exponent is capped at $1$,
so $\rlct(f)=1$. For all $a,b\ge 2$, one has
$\frac{1}{a}+\frac{1}{b}\le 1$, with equality if and only if $a=b=2$.

\medskip

We confirm the result using $\rlct(I)=\min_i(k_i+1)/m_i$.
Apply the $(b,a)$-weighted blow-up of $\mathbb{C}^2$ (with weights
$b$ on the $x$-axis and $a$ on the $y$-axis). The resulting
exceptional divisor $E$ satisfies
\[
  m_E = \nu_E(f)
  = \min\bigl\{\nu_E(x^a),\,\nu_E(y^b)\bigr\}
  = ab,
  \qquad
  k_E = a+b-1.
\]
The equalities $\nu_E(x)=b$ and $\nu_E(y)=a$ follow from the weight vector,
hence $\nu_E(x^a)=ab$ and $\nu_E(y^b)=ab$. The discrepancy
$k_E=a+b-1$ is the standard formula for a $(b,a)$-weighted blow-up in
$\mathbb{C}^2$ (see, e.g., Igusa~\cite[Ch.~5, §5.4]{Igusa00}). Therefore
\[
  \rlct(f) = \frac{k_E+1}{m_E} = \frac{a+b}{ab}
  = \frac{1}{a}+\frac{1}{b},
\]
in agreement with the scaling computation above.

\medskip

The compact edge of $\Gamma(f)$ has primitive inward normal
$w=(b,a)/\gcd(a,b)$, and the weighted blow-up with this primitive weight
extracts a \emph{single} divisor $E$ of minimal spectral value $1/a+1/b$;
the subdivision rays needed to complete the toric resolution carry strictly
larger values. Hence the leading value $\lambda_1=1/a+1/b$ is attained, in
this resolution, by one exceptional divisor, and
\[
m_{\lambda_1}=1
\qquad\text{whenever no two divisors of value }\lambda_1\text{ meet.}
\]
(The leading multiplicity exceeds $1$ only when several components of value
$\lambda_1$ genuinely cross in a chart, as in Example~\ref{ex:log}; this is
not measured by $\gcd(a,b)$. For instance $x^3+y^3$ has $\gcd=3$ but
$m_{\lambda_1}=1$, since the three branches of $\{f=0\}$ meet $E$ at distinct
points, whereas the node $x^2+y^2$ has $m_{\lambda_1}=2$ through the
crossing of its two branches.)

\medskip

In the coprime case $\gcd(a,b)=1$ the weighted blow-up produces a single
exceptional divisor and no branch crossing of equal value, so
\[
\Lambda_\pi(f)=\Bigl\{\frac{1}{a}+\frac{1}{b}\Bigr\},
\qquad
m_{\lambda_1}=1.
\]
Accordingly, the volumetric asymptotics takes the form
\[
V(\varepsilon)
=
c_1\,\varepsilon^{2(1/a+1/b)}
+
o\bigl(\varepsilon^{2(1/a+1/b)}\bigr),
\]
with no logarithmic correction. The constant is
\[
c_1=
\Vol\bigl\{(u,v)\in\mathbb{C}^2:\ |u^a+v^b|<1\bigr\}>0.
\]
\end{example}

\begin{example}
\label{ex:newton}

Let
\[
f(x,y)=x^4+x^2y^3+y^7\in\mathcal{O}_2.
\]
The Newton polygon $\Gamma(f)$ has vertices
$(4,0)$, $(2,3)$ and $(0,7)$, with compact edges
$\sigma_1$ joining $(4,0)$ to $(2,3)$ and
$\sigma_2$ joining $(2,3)$ to $(0,7)$.

\begin{figure}[ht]
\centering
\begin{tikzpicture}[scale=0.72]
  \fill[blue!10] (4,0) -- (2,3) -- (0,7) -- (0,8) -- (5.2,8) -- (5.2,0) -- cycle;

  \foreach \x in {0,...,5}
    \foreach \y in {0,...,8}
      \fill[gray!30] (\x,\y) circle (1.5pt);

  \draw[->] (-0.3,0) -- (5.5,0) node[right] {$i$};
  \draw[->] (0,-0.3) -- (0,8.5) node[above] {$j$};

  \foreach \x in {1,...,5}
      \draw (\x,2pt) -- (\x,-2pt)
      node[below,font=\tiny] {\x};

  \foreach \y in {1,...,8}
      \draw (2pt,\y) -- (-2pt,\y)
      node[left,font=\tiny] {\y};

  \draw[very thick,blue] (4,0)--(2,3)--(0,7);

  \fill[red] (4,0) circle (3pt)
      node[below right,black] {$(4,0)$};

  \fill[red] (2,3) circle (3pt)
      node[above right,black] {$(2,3)$};

  \fill[red] (0,7) circle (3pt)
      node[left,black] {$(0,7)$};

  \node[blue] at (3.25,1.8) {$\sigma_1$};
  \node[blue] at (1.3,5.2) {$\sigma_2$};
\end{tikzpicture}
\end{figure}

The face polynomials are

\[
f_{\sigma_1}=x^4+x^2y^3,
\qquad
f_{\sigma_2}=x^2y^3+y^7.
\]

A direct computation shows that neither face polynomial has a critical
point in $(\mathbb C^*)^2$; hence $f$ is Newton
non-degenerate and Varchenko's formula applies
\cite[§0.4--§0.6]{Varchenko76}.

\medskip

The associated toric log resolution has two exceptional divisors
corresponding to the primitive inward normals

\[
w_1=(3,2),
\qquad
w_2=(2,1).
\]

For a toric valuation with weight $(p,q)$,

\[
k=p+q-1.
\]

For $E_1$ one obtains

\[
m_1
=
\min\{12,12,14\}
=
12,
\qquad
k_1=4,
\]

and therefore

\[
\lambda_1
=
\frac{k_1+1}{m_1}
=
\frac5{12}.
\]

Similarly,

\[
m_2
=
\min\{8,7,7\}
=
7,
\qquad
k_2=2,
\]

so

\[
\lambda_2
=
\frac{k_2+1}{m_2}
=
\frac37.
\]

Hence

\[
\Lambda_\pi(f)
=
\left\{
\frac5{12},
\frac37
\right\},
\qquad
\rlct(f)
=
\frac5{12},
\]

and both leading candidates are produced by a single divisor, so

\[
m_{5/12}
=
m_{3/7}
=
1.
\]

\medskip
\noindent
\emph{A further candidate and the distinction between candidate and
visible spectra.}

This resolution is not of the type considered in
Remark~\ref{rem:when-complete}. In the chart

\[
x=\alpha\beta^2,
\qquad
y=\beta,
\]

one computes

\[
f
=
\beta^7
\bigl[(\alpha^2+1)+\alpha^4\beta\bigr].
\]

Thus the local unit depends nontrivially on the exceptional coordinate,

\[
u(\alpha,\beta)
=
(\alpha^2+1)+\alpha^4\beta,
\]

since

\[
\nu_{E_2}(x^4)=8=m_2+1.
\]

Consequently,

Definition~\ref{def:spectrum-extended}
produces the additional candidate

\[
\frac{k_2+2}{m_2}
=
\frac47,
\]

so that

\[
\frac47
\in
\widetilde{\Lambda}_\pi(f).
\]

At the level of local charts this candidate produces nontrivial residue
contributions. Local residue computations suggest that, in this example,
the residue contributions arising from the two charts meeting at $4/7$
appear with opposite signs. Determining whether they cancel
globally requires the complete residue construction obtained after
gluing the local charts, which lies beyond the scope of the present
paper.

Accordingly, we neither assert that

\[
\frac47\in\Lambda(f)
\]

nor that it is absent from $\Lambda(f)$.

Rather, this example illustrates the conceptual distinction between the
extended candidate spectrum and the visible spectrum: the
former is determined locally by the resolution data, whereas the latter
consists precisely of those candidate values whose global residue does
not vanish.

\medskip

This issue does not affect the leading asymptotic behaviour.
Indeed,

\[
\frac47>\frac37>\frac5{12},
\]

so the first two terms of the expansion are completely determined by
$E_1$ and $E_2$.

Therefore

\[
V(\varepsilon)
=
c_1\varepsilon^{5/6}
+
c_2\varepsilon^{6/7}
+
o(\varepsilon^{6/7}),
\]

with

\[
c_1,c_2>0.
\]

Any contribution arising from the candidate value $4/7$, if present,
would be of order

\[
O(\varepsilon^{8/7})
=
o(\varepsilon^{6/7}),
\]

precisely because $4/7>3/7$, so it is absorbed into the remainder
regardless of how the cancellation question above is ultimately resolved.

\medskip

Applying
Theorem~\ref{thm:volume-reconstruction},

\[
\rlct(f)
=
\frac12
\lim_{\varepsilon\to0}
\frac{\log V(\varepsilon)}
{\log\varepsilon}
=
\frac5{12}.
\]

After subtracting the leading term,

\[
R_2(\varepsilon)
=
V(\varepsilon)
-
c_1\varepsilon^{5/6},
\]

one obtains

\[
\lambda_2
=
\frac12
\lim_{\varepsilon\to0}
\frac{\log|R_2(\varepsilon)|}
{\log\varepsilon}
=
\frac37.
\]

Since

\[
m_{5/12}
=
m_{3/7}
=
1,
\]

no logarithmic factors occur in the first two terms.

This example illustrates that the extended candidate spectrum may
contain values whose visibility depends on global residue
cancellations. It is precisely this phenomenon that motivates the
distinction between the extended candidate spectrum
$\widetilde{\Lambda}_\pi(f)$ and the visible spectrum
$\Lambda(f)$.
\end{example}

\begin{example}[A Fermat singularity: a simple leading pole]
\label{ex:fermat}

Consider
\[
f(x,y,z)=x^4+y^4+z^4\in\mathcal O_3.
\]
This example illustrates that a naive lattice-point count on the Newton
facet does \emph{not} determine the multiplicity of the leading pole.

The Newton polyhedron has a single compact facet with vertices
$(4,0,0)$, $(0,4,0)$ and $(0,0,4)$, whose primitive inward normal is
$(1,1,1)$. Since
\[
\nabla(x^4+y^4+z^4)=0
\]
has no solution in $(\mathbb C^*)^3$, the polynomial is Newton
non-degenerate.

The corresponding toric modification is the ordinary blow-up
\[
\pi:Y\longrightarrow\mathbb C^3.
\]
In the chart
\[
x=t,\qquad y=tv,\qquad z=tw,
\]
one has
\[
f\circ\pi
=
t^4(1+v^4+w^4).
\]
Hence the unique exceptional divisor satisfies
\[
m_E=4,\qquad
k_E=2,\qquad
\lambda_E=\frac{k_E+1}{m_E}=\frac34.
\]

The strict transform is smooth and meets the exceptional divisor
transversely. Thus $\pi$ is already a log resolution and

\[
\Lambda_\pi(f)
=
\Bigl\{\frac34,1\Bigr\},
\qquad
\rlct(f)=\frac34,
\qquad
m_{3/4}=1.
\]

Since only one divisor contributes to the value $3/4$, the pole is
simple and no logarithmic factor appears. Accordingly,

\[
V(\varepsilon)
\asymp
\varepsilon^{3/2}.
\]

Indeed, the scaling
\[
x=\varepsilon^{1/4}u,\qquad
y=\varepsilon^{1/4}v,\qquad
z=\varepsilon^{1/4}w
\]
gives
\[
f=\varepsilon(u^4+v^4+w^4),
\]
while
\[
dV_{\mathbb R^6}
=
\varepsilon^{3/2}\,dV_{uvw}.
\]

The lattice-point count
\[
\#\{(a,b,c)\in\mathbb Z_{>0}^3:a+b+c=4\}=3
\]
computes the geometric genus of the surface germ, not the multiplicity
of the pole. The order of the pole is determined instead by the maximal
number of divisors with the same spectral value meeting in one chart.
Here this number is one, so no logarithmic factor occurs.

Finally, Theorem~\ref{thm:volume-reconstruction} recovers
\[
\rlct(f)
=
\frac34,
\qquad
m_{3/4}=1,
\]
directly from the asymptotic behaviour of
$V(\varepsilon)$.

\end{example}

\begin{example}[Normal crossings: a genuine logarithmic factor]
\label{ex:log}

Consider
\[
f(x,y)=xy\in\mathcal O_2.
\]
Since the divisor $\{xy=0\}$ already has simple normal crossings, no
blow-up is required. Its two irreducible components satisfy

\[
m_1=m_2=1,
\qquad
k_1=k_2=0,
\qquad
\lambda_1=\lambda_2=1.
\]

Because the two divisors meet at the origin, the local zeta function has
a pole of order two at $s=1$. Consequently,

\[
V(\varepsilon)
=
\Vol\{|xy|<\varepsilon\}
\asymp
\varepsilon^2|\log\varepsilon|,
\qquad
\varepsilon\to0^+.
\]

Thus

\[
\Lambda(xy)=\{1\},
\qquad
m_1=2,
\]

and the asymptotic expansion contains the logarithmic factor

\[
|\log\varepsilon|^{m_1-1}
=
|\log\varepsilon|.
\]

This contrasts with Example~\ref{ex:fermat}. There the leading spectral
value is produced by a single divisor, whereas here two divisors with
the same spectral value intersect in one chart, giving rise to a genuine
higher-order pole and hence to the logarithmic term.

Applying
Theorem~\ref{thm:volume-reconstruction},
both the exponent $\lambda=1$ and its multiplicity $m_1=2$ are recovered
directly from the asymptotic expansion of the volume.

\end{example}

\section*{Acknowledgments}

We thank Guillem Blanco for his careful reading of an earlier version of
this paper and, in particular, for pointing out a gap in the original proof
of Proposition~\ref{prop:poles}. His observation led directly to the
corrected treatment of the candidate and visible spectra presented in
Section~\ref{sec:vol-spectrum}, including the distinction between candidate
and actual poles, which is now emphasized throughout the paper.


\begin{thebibliography}{99}



\bibitem{ArnoldGuseinVarchenko85}
Arnold, V.I., Gusein-Zade, S.M., Varchenko, A.N.: Singularities of Differentiable Maps, Vols. II. Birkhäuser, Boston (1985). \href{https://doi.org/10.1007/978-0-8176-8343-6}{https://doi.org/10.1007/978-0-8176-8343-6}

\bibitem{Atiyah70}
Atiyah, M.F.: Resolution of singularities and division of distributions. Comm. Pure Appl. Math. \textbf{23}, 145–150 (1970). \href{https://doi.org/10.1002/cpa.3160230202}{https://doi.org/10.1002/cpa.3160230202}

\bibitem{BM97}
Bierstone, E., Milman, P.D.: Canonical desingularization in characteristic zero by blowing up the maximum strata of a local invariant. Invent. Math. \textbf{128}, 207–302 (1997). \href{https://doi.org/10.1007/s002220050141}{https://doi.org/10.1007/s002220050141}

\bibitem{BiviAFukui18}
Bivà-Ausina, C., Fukui, T.: Mixed Łojasiewicz exponents, log canonical thresholds of ideals and bi-Lipschitz equivalence. J. Math. Soc. Japan \textbf{70}(3), 1045–1070 (2018). \href{https://doi.org/10.1016/j.jpaa.2015.06.007}{https://doi.org/10.1016/j.jpaa.2015.06.007}

\bibitem{BernigLytchak07}
Bernig, A., Lytchak, A.: Tangent spaces and Gromov–Hausdorff limits of subanalytic sets. J. Reine Angew. Math. \textbf{605}, 1–20 (2007). \href{https://doi.org/10.1515/CRELLE.2007.050}{https://doi.org/10.1515/CRELLE.2007.050}

\bibitem{Blum16}
Blum, H., Jonsson, M.: Thresholds, valuations, and K-stability. Adv. Math. \textbf{371}, paper No. 107611, 46 pp. (2020). \href{https://doi.org/10.1016/j.aim.2020.107062}{https://doi.org/10.1016/j.aim.2020.107062}

\bibitem{Collins18}
Collins, T.C.: Log-canonical thresholds in real and complex dimension two. Ann. Inst. Fourier (Grenoble) \textbf{68}(7), 2883–2900 (2018). \href{https://doi.org/10.5802/aif.3229}{https://doi.org/10.5802/aif.3229}

\bibitem{Demailly}
Demailly, J.-P.: Complex Analytic and Differential Geometry. Online book, available at the author’s webpage \href{http://www-fourier.ujf-grenoble.fr/~demailly/books.html}{http://www-fourier.ujf-grenoble.fr/~demailly/books.html} (2023). Accessed 20 March 2026

\bibitem{Denef91}
Denef, J.: Report on Igusa's local zeta function. Astérisque \textbf{201-202-203}, S\'eminaire Bourbaki exp.\ n$^{\text{o}}$ 741, 359--386 (1991). \href{http://www.numdam.org/item?id=SB_1990-1991__33__359_0}{http://www.numdam.org/item?id=SB\_1990-1991\_\_33\_\_359\_0}

\bibitem{DenefLoeser92}
Denef, J., Loeser, F.: Caractéristiques d'Euler–Poincaré, fonctions zêta locales et modifications analytiques. J. Amer. Math. Soc. \textbf{5}(4), 705–720 (1992). \href{https://doi.org/10.2307/2152708}{https://doi.org/10.2307/2152708}

\bibitem{DenefLoeser99}
Denef, J., Loeser, F.: Germs of arcs on singular algebraic varieties and motivic integration. Invent. Math. \textbf{135}(1), 201–232 (1999). \href{https://doi.org/10.1007/s002220050284}{https://doi.org/10.1007/s002220050284}

\bibitem{EinLazMust04}
Ein, L., Lazarsfeld, R., Mustaţă, M.: Contact loci in arc spaces. Compos. Math. \textbf{140}(5), 1229–1244 (2004). \href{https://doi.org/10.1112/S0010437X04000429}{https://doi.org/10.1112/S0010437X04000429}

\bibitem{DAH}
Grulha, N.: Divisorial Persistence and Asymptotic Homology of Analytic Pairs. Preprint.

\bibitem{Hu15}
Hu, Z.: Valuations and Log Canonical Thresholds. Pure Appl. Math. Q. \textbf{11}(1), 49–86 (2015). \href{https://doi.org/10.4310/PAMQ.2015.v11.n1.a3}{https://doi.org/10.4310/PAMQ.2015.v11.n1.a3}

\bibitem{Igusa00}
Igusa, J.: An Introduction to the Theory of Local Zeta Functions. AMS/IP Studies in Advanced Mathematics, vol. 14. American Mathematical Society, Providence, RI (2000). \href{https://doi.org/10.1090/amsip/014}{https://doi.org/10.1090/amsip/014}

\bibitem{Kollar13}
Kollár, J.: Singularities of the Minimal Model Program. Cambridge Tracts in Mathematics, vol. 200. Cambridge University Press, Cambridge (2013). \href{https://doi.org/10.1017/CBO9781139547895}{https://doi.org/10.1017/CBO9781139547895}

\bibitem{KollarMori98}
Kollár, J., Mori, S.: Birational Geometry of Algebraic Varieties. Cambridge Tracts in Mathematics, vol. 134. Cambridge University Press, Cambridge (1998). \href{https://doi.org/10.1017/CBO9780511662560}{
https://doi.org/10.1017/CBO9780511662560}

\bibitem{Korevaar04}
Korevaar, J.: Tauberian Theory: A Century of Developments. Grundlehren der mathematischen Wissenschaften, vol. 322. Springer, Berlin (2004). \href{https://doi.org/10.1007/978-3-662-10225-1}{https://doi.org/10.1007/978-3-662-10225-1}

\bibitem{Lazarsfeld04}
Lazarsfeld, R.: Positivity in Algebraic Geometry II. Ergebnisse der Mathematik und ihrer Grenzgebiete, 3. Folge, vols. 48/49. Springer-Verlag, Berlin (2004). \href{https://doi.org/10.1007/978-3-642-18810-7}{https://doi.org/10.1007/978-3-642-18810-7}

\bibitem{LJ-Teissier08}
Lejeune-Jalabert, M., Teissier, B.: Clôture intégrale des idéaux et équisingularité. Ann. Fac. Sci. Toulouse Math. (6) \textbf{17}(4), 781–859 (2008). \href{https://doi.org/10.5802/afst.1203}{https://doi.org/10.5802/afst.1203}

\bibitem{Mustata06}
Mustaţă, M.: IMPANGA lecture notes on log canonical thresholds, notes by T. Szemberg. In: Contributions to Algebraic Geometry, IMPANGA Lecture Notes (eds. Szemberg T., et al.), pp. 407–442. Eur. Math. Soc., Zürich (2012). \href{https://doi.org/10.4171/114-1/16}{https://doi.org/10.4171/114-1/16}

\bibitem{Mustata01}
Mustaţă, M.: Jet schemes of locally complete intersection canonical singularities. Invent. Math. \textbf{145}(2), 397–424 (2001). \href{https://doi.org/10.1007/s002220100152}{https://doi.org/10.1007/s002220100152}

\bibitem{Hir64}
Hironaka, H.: Resolution of singularities of an algebraic variety over a field of characteristic zero. I, II. Ann. Math. (2) \textbf{79}, 109–203; 205–326 (1964). \

\bibitem{Saito21}
Saito, M.: On real log canonical thresholds. Preprint, arXiv:2108.01923 (2021).

\bibitem{Valette07}
Valette, A.: The link of the germ of a semi-algebraic metric space. Proc. Amer. Math. Soc. \textbf{135}(10), 3083–3090 (2007). \href{https://doi.org/10.1090/S0002-9939-07-08878-8}{https://doi.org/10.1090/S0002-9939-07-08878-8}

\bibitem{Varchenko76}
Varchenko, A.N.: Asymptotic behavior of integrals and the Newton polyhedron. Invent. Math. \textbf{37}, 253–262 (1976). 

\bibitem{Watanabe09}
Watanabe, S.: Algebraic Geometry and Statistical Learning Theory. Cambridge Monographs on Applied and Computational Mathematics, vol. 25. Cambridge University Press, Cambridge (2009). \href{https://doi.org/10.1017/CBO9780511800474}{https://doi.org/10.1017/CBO9780511800474}

\bibitem{Watanabe2024}
Watanabe, S.: Recent Advances in Algebraic Geometry and Bayesian Statistics. Information Geometry \textbf{7}(Suppl.~1), S187--S209 (2024). \href{https://doi.org/10.1007/s41884-022-00083-9}{https://doi.org/10.1007/s41884-022-00083-9}

\bibitem{Wei23}
Wei, J., Murfet, D., Gong, J., Li, B., Gell-Redman, J., Quella, M.: Deep learning is singular, and that's good. IEEE Trans. Neural Networks Learn. Syst. \textbf{34}(12), 10473–10486 (2023). \href{https://doi.org/10.1109/TNNLS.2022.3167409}{https://doi.org/10.1109/TNNLS.2022.3167409}

\bibitem{deFEI08}
de Fernex, T., Ein, L., Ishii, S.: Divisorial valuations via arcs. Publ. Res. Inst. Math. Sci. \textbf{44}(2), 425–448 (2008). \href{https://doi.org/10.2977/prims/1210167333}{https://doi.org/10.2977/prims/1210167333}

\end{thebibliography}
\end{document}